\documentclass[12pt]{amsart}

\usepackage{amsmath, graphicx, cite, enumerate, comment}
\usepackage{amssymb, amsthm, amscd}
\usepackage{latexsym}
\usepackage{amssymb}
\usepackage[english]{babel}
\usepackage{amsmath,amssymb}
\usepackage{amsfonts}
\usepackage{amscd}
\usepackage{cite}
\usepackage{marvosym}
\usepackage{mathrsfs}
\usepackage{amsthm}
\usepackage{hyperref}
\usepackage{verbatim}

\usepackage{tikz} \usetikzlibrary{arrows,shapes,patterns} 
\usepackage{lipsum} 

\newtheorem*{thmA}{Theorem A}
\newtheorem*{thmB}{Corollary B}
\newtheorem*{thmC}{Theorem C}
\newtheorem*{thmD}{Remark D}
\newtheorem*{thmE}{Corollary E}
\newtheorem*{thmF}{Theorem F}
\newtheorem*{thmG}{Remark G}

\newtheorem{thm}{Theorem}
\newtheorem{lemm}[thm]{Lemma}

\newtheorem{prop}[thm]{Proposition}

\theoremstyle{remark}
\newtheorem{rmk}[thm]{Remark}

\theoremstyle{definition}

\numberwithin{equation}{section}

\newcommand{\R}{\mathbb{R}}

\newcommand{\be}{\begin{equation}}
\newcommand{\ee}{\end{equation}}
\newcommand{\bd}{\begin{displaymath}}
\newcommand{\ed}{\end{displaymath}}

     \title[Index estimates for free boundary minimal hypersurfaces]{Index estimates for free boundary minimal hypersurfaces}
     \author{Lucas Ambrozio, Alessandro Carlotto and Ben Sharp}
     \address{ \noindent L. Ambrozio: Imperial College London, South Kensington Campus, London SW7 2AZ, United Kingdom, \textit{E-mail address: l.ambrozio@imperial.ac.uk} \newline \newline \indent A. Carlotto: ETH Inst. f\"ur Theoretische Studien, Clausiusstrasse 47, 8092 Z\"urich, Switzerland \textit{E-mail address: alessandro.carlotto@eth-its.ethz.ch} \newline \newline \indent B. Sharp: SNS Pisa, Piazza dei Cavalieri 7, 56126 Pisa, Italy, \textit{E-mail address: benjamin.sharp@sns.it} }
     
\begin{document}
     	
     	\begin{abstract}
     		We show that the Morse index of a properly embedded free boundary minimal hypersurface in a strictly mean convex domain of the Euclidean space grows linearly with the dimension of its first relative homology group (which is at least as big as the number of its boundary components, minus one). In ambient dimension three, this implies a lower bound for the index of a free boundary minimal surface which is linear both with respect to the genus and the number of boundary components. Thereby, the compactness theorem by Fraser and Li implies a strong compactness theorem for the space of free boundary minimal surfaces with uniformly bounded Morse index inside a convex domain. Our estimates also imply that the examples constructed, in the unit ball, by Fraser-Schoen and Folha-Pacard-Zolotareva have arbitrarily large index. Extensions of our results to more general settings (including various classes of positively curved Riemannian manifolds and other convexity assumptions) are discussed.
     	\end{abstract}
     	
     	\maketitle       
							
	\section{Introduction} 
		
		\indent Given $(\Omega^{n+1},g)$ a smooth Riemannian manifold with boundary, we shall be concerned here with certain global properties of free boundary minimal hypersurfaces $M^n\subset \Omega^{n+1}$, namely hypersurfaces that are critical points of the area functional when the boundary $\partial M$ is not fixed (like in Plateau's problem) but subject to the sole constraint $\partial M\subset\partial \Omega$. 
		Due to their self-evident geometric interest (which can be traced back at least to Courant \cite{Cou}), these variational objects have been widely studied and a number of existence results have been obtained via surprisingly diverse methods (see, among others, 
		\cite{MeeYau, Str, GruJos1, Jos2, Li, MaxNunSmi , FraSch3, FolPacZol, FreGulMcg, Ye} and references therein). Free boundary minimal hypersurfaces also naturally arise in partitioning problems for convex bodies, in capillarity problems for fluids and, as has significantly emerged in recent years, in connection to extremal metrics for Steklov eigenvalues for manifolds with boundary (see primarily the works by Fraser-Schoen \cite{FraSch1, FraSch2, FraSch3} and references therein). From an analytic perspective, it should also be mentioned that their boundary regularity has been the object of extensive investigations (let us mention, for instance, \cite{Jos1, GruJos2, HilNit1, HilNit2}). \\
		\indent The results we are about to present regard  the comparison between the Morse index and the topology of free boundary minimal hypersurfaces. Roughly speaking, the index is a non-negative integer measuring the maximal number of distinct deformations that locally decrease the area to second-order (subject to the aforementioned constraint $\partial M\subset \partial\Omega$). On the other hand, we shall describe the topology of a manifold with boundary by means of its (real) homology groups. As is well-known, in the most basic case of orientable surfaces with boundary the topological type can be completely described by means of two numbers, namely the genus and the number of boundary components of the surface in question. \\
		\indent There are some general results about the geometry and topology of stable (= index zero) and index one compact free boundary minimal surfaces in general three-manifolds whose boundary satisfies some convexity assumption. For example, it is known that stable compact two-sided free boundary minimal surfaces in mean convex domains of three-manifolds with non-negative scalar curvature must be topological disks or totally geodesic annuli (see for example \cite{CheFraPan}). Moreover, Cheng, Fraser and Pang showed in the same article that there exists an explicit upper bound on the genus and the number of boundary components of index one compact two-sided free boundary minimal surfaces in such manifolds. Related results about the topology of free boundary volume-preserving stable CMC surfaces in strictly mean convex domains of the three-dimensional Euclidean space were obtained by Ros in \cite{Ros2}.\\
		\indent On the other hand, Fraser and Schoen \cite{FraSch3} have proven that if $M^n\subset B^{n+1}$ (the unit ball in $\mathbb{R}^{n+1}$) then either $M^n$ is a flat disk (whose index is one) or its Morse index is at least $n+2$. We also remark that some interesting results on the index of free boundary minimal submanifolds of higher codimension have been proven in \cite{Fra} and \cite{FraSch3}, Theorem 3.1. In this paper, we apply the techniques developed in \cite{AmbCarSha} (but see also \cite{Ros2}) to prove a general lower bound for the index of free boundary minimal hypersurfaces in terms of topological data of the hypersurface in question.

		For the sake of simplicity,  in this introduction we shall state our results in the special case of domains of the Euclidean space, while the corresponding extensions to Riemannian manifolds satisfying certian curvature conditions are postponed to the last section of this paper (see Theorem \ref{thm:genamb}, Theorem \ref{thm:combgen} and related comments). \\
			\indent Our first main result is the following.
		
		\begin{thmA}
		Let $\Omega^{n+1}$ be a strictly mean convex domain of the $(n+1)$-dimensional Euclidean space, $n\geq 2$. Let $M^n$ be a compact, orientable, properly embedded free boundary minimal hypersurface in $\Omega$. Then
		\begin{equation*}
		index(M) \geq \frac{2}{n(n+1)} dim H_{1}(M,\partial M;\mathbb{R}).
		\end{equation*}				 		 		
		\end{thmA}
		In the above inequality, $H_{1}(M,\partial M;\mathbb{R})$ denotes the first relative homology group with real coefficients. The dimension of this homology group can be explicitly computed in terms of the homology groups of $M^n$ and $\partial M$ (see Lemma \ref{lem:relhom}). In particular, we obtain an estimate for the index in terms of the number of boundary components. 
		\begin{thmB}
		Let $\Omega^{n+1}$ be a strictly mean convex domain of the $(n+1)$-dimensional Euclidean space, $n\geq 3$. Let $M^n$ be a compact, orientable, properly embedded free boundary minimal hypersurface in $\Omega$ with $r\geq 1$ boundary components. Then
		\begin{equation*}
		 index(M) \geq \frac{2}{n(n+1)} (r-1).
		\end{equation*}				 		 		
		\end{thmB}
		
		In the case of free boundary minimal surfaces ($n=2$), the estimate also involves the genus of the surface (Lemma \ref{lem:twodimhom}) and can in fact be upgraded to the more general scenario when the ambient domain is only \textsl{weakly} mean convex. This requires an \textsl{ad hoc} argument, and exploits a result of Ros \cite{Ros1}.

		\begin{thmC}
		 Let $\Omega^3$ be a mean convex domain of the three-dimensional Euclidean space. Let $M^2$ be a compact, orientable, properly embedded free boundary minimal surface in $\Omega$ with genus $g$ and $r\geq 1$ boundary components. Then
		\begin{equation*}
		 index(M) \geq \frac{1}{3} (2g+r-1) .
		\end{equation*}	 		
		\end{thmC}
		\indent Let us remark that the conclusion of Theorem C coincides with the one obtained by Ros and Vergasta \cite{RosVer} in the special case of index one free boundary minimal surfaces in strictly convex domains of $\mathbb{R}^{3}$ (notice that there are no stable free boundary minimal surfaces in such domains). Furthermore, by following the computations of Savo \cite{Sav} and in the sole case of strictly convex bodies, the conclusion of Theorem C has been obtained independently by Sargent in \cite{Sar}. Again under a strict convexity condition, Sargent obtained the conclusion of our Theorem F, stated below, when $\alpha=0$.


		\begin{thmD} The above theorem can be used to understand the behaviour of the index of some known examples of free boundary minimal surfaces constructed in the unit ball in $\mathbb{R}^3$. In particular, the examples constructed by Fraser and Schoen \cite{FraSch3}, which have genus zero and an arbitrary number of boundary components, and the examples constructed by Folha, Pacard and Zolotareva \cite{FolPacZol}, which have genus one and an arbitrarily large number of boundary components, have their Morse indices growing linearly with the number of boundary components. 
		\end{thmD}	
		\indent Another corollary that can be deduced from the above estimate is a compactness theorem for free boundary minimal surfaces with bounded index in \textsl{strictly convex} domains. In fact, Fraser and Li proved that in those domains the set of compact, properly embedded free boundary minimal surfaces with uniformly bounded genus and number of boundary components is strongly compact  (Theorem 1.2 in \cite{FraLi}). In particular, our index estimate shows that the following statement is actually \textit{equivalent} to their compactness result.
		\begin{thmE}
		Let $\Omega^3$ be a compact domain in $\mathbb{R}^3$ whose boundary is strictly convex. Then any sequence $\{M_i^2\}$ of compact, properly embedded free boundary minimal surfaces in $\Omega$ that has uniformly bounded index has a subsequence converging smoothly and graphically to a compact properly embedded free boundary minimal surface $M^2$ in $\Omega$.
		\end{thmE}
		
		\indent Lastly, we shall present here a variation on Theorem A which holds true for (strictly) two-convex domains of the Euclidean space. Let us recall that, if the second fundamental form $II^{\partial \Omega}$ is defined with respect to the outward unit normal of $\partial \Omega$ then two-convexity is equivalent to the requirement that the sum of \emph{any} two eigenvalues of $II^{\partial \Omega}$ be strictly positive.
		For this class of domains, we prove the following:
		
		\begin{thmF}
			Let $\Omega^{n+1}$ be a strictly two-convex domain of the $(n+1)$-dimensional Euclidean space. Let $M^n$ be a compact properly embedded free boundary minimal hypersurface of $\Omega$. Then, for any $\alpha \in [0,1]$ we have 
				\begin{equation*}
				index(M) \geq \frac{2}{n(n+1)}\left(\alpha dim H_1 (M,\partial M;\R) + (1-\alpha) dim H_{n-1}(M,\partial M;\R)\right)	\end{equation*}
				hence, as a special case
				\begin{equation*}
				index(M) \geq \frac{1}{n(n+1)}\left( dim H_1 (M,\partial M;\R) + dim H_{n-1}(M,\partial M;\R)\right).	\end{equation*}
				\end{thmF}
				
				\begin{thmG}
					 All the main results in this paper (with the sole exception of Corollary E, which relies on the compactness statement by Fraser and Li), including those in Section \ref{app:general}, actually hold for properly \textsl{immersed} free boundary minimal hypersurfaces. The necessary modifications to our proofs are of purely notational character.
			\end{thmG}

				\indent The paper is organised as follows: some preliminary facts are recalled in Section \ref{sec:indexfb} (concerning the Morse index for free boundary minimal hypersurfaces) and in Section \ref{sec:hodge} (concerning some Hodge-theoretic aspects for manifolds with boundary), while the core of our approach (following \cite{AmbCarSha}) is presented in Section \ref{sec:testf} and Section \ref{sec:mainproof}, the latter devoted to completing the proofs of Theorem A, Theorem C and Theorem F. The case of general ambient manifolds with special cases of particular interest is presented in Section \ref{app:general}.   \\
		 	
		\noindent \textit{Acknowledgements:} The authors would like to thank Ivaldo Nunes, Fernando Cod\'a Marques and Andr\'e Neves for their interest in this work. L. A. is supported by the ERC Start Grant PSC and LMCF 278940 and would like to thank the Scuola Normale Superiore where part of this project was completed. This article was done while  A. C. was an ETH-ITS fellow: the outstanding support of Dr. Max R\"ossler, of the Walter Haefner Foundation and of the ETH Z\"urich Foundation are gratefully acknowledged. B.S. would like to thank the ETH-FIM for their hospitality and excellent working environment during the completion of this project. B.S. was partially supported by the Scuola Normale Superiore (Commissione Ricerca, Progetto Giovani Ricercatori).
	
	\section{The index of free boundary minimal hypersurfaces}\label{sec:indexfb}

	\indent Let $(\Omega^{n+1},g)$ be a smooth, orientable Riemannian manifold with boundary $\partial \Omega$. We say that a compact, connected, embedded hypersurface $M^n$ in $\Omega^{n+1}$ is properly embedded if it has no interior points touching $\partial \Omega$, i.e., $M \cap \partial \Omega = \partial M$. Throughout this paper, we always tacitly assume that $M^n$ is itself orientable (hence, equivalently, two-sided) and choose a unit vector field $N$ normal to $M^n$ (the one-sided case can be dealt with as in \cite{AmbCarSha}, Section 2). Let us remark that this is always the case if $\Omega^{n+1}$ is simply-connected (e. g. for convex domains in $\mathbb{R}^{n+1}$). We shall denote by $\nu$ the outward pointing conormal of $\partial M$, i.e., the unique unit vector field on $\partial M$ that is tangent to $M^n$,  normal to $\partial M$ and points outside $M^n$. \\
	\indent When considering the area functional restricted to this class of hypersurfaces, the allowed variations are produced by flows $\psi_t$ of vector fields $X$ on $\Omega^{n+1}$ that are tangent to $\partial \Omega$. The first variation formula of area for such an admissible variation is given by 
	\begin{equation*}
	\frac{d}{dt}_{|t=0}|\psi_{t}(M)| = \int_{M}H^{M}g(N,X) d\mu + \int_{\partial M} g(\nu,X) d\sigma.
	\end{equation*}
	\indent It follows that critical points are minimal ($H^{M}=0$) and intersect $\partial \Omega$ orthogonally ($\nu \perp \partial \Omega$). The last condition is known as the free boundary property of $M^n$. \\
	\indent  Given \textit{any} smooth function $\phi$ on $M^n$, there exists an admissible vector field $X$ such that $X=\phi N$ on $M^n$ (see for example \cite{MaxNunSmi}, Section 2). The second variation of area at a free boundary minimal hypersurface $M^n$ along the flow of the vector fields considered above defines the quadratic form
	\begin{equation*}
	Q^M(\phi,\phi) := \frac{d^2}{dt^2}_{|t=0}|\psi_{t}(M)| = \int_{M} |\nabla^{M}\phi|^2 - (Ric^{\Omega}(N,N) + |A|^2)\phi^2d\mu - \int_{\partial M} \langle D_{N}\nu,N \rangle \phi^2d\sigma.
	\end{equation*}
	\indent In the above formula, $Ric^{\Omega}$ denotes the Ricci tensor of $\Omega$, $A$ denotes the second fundamental form of $M^n$ and $D$ the covariant derivative in $(\Omega,g)$. Notice that, since $M^n$ is free boundary, $N$ is tangent to $\partial \Omega$ and the term $\langle D_{N}\nu,N \rangle$ is precisely the second fundamental form $II^{\partial \Omega}$ of the boundary of the domain with respect to $\nu$ applied to the vector field $N$. \\
	\indent The quadratic form $Q^M$ is called the index form of the free boundary minimal surface $M^n$. The index of $M^n$ is defined as the index of $Q^{M}$, that is, the maximal dimension of a linear subspace $V$ in $C^{\infty}(M)$ such that $Q^M(\phi,\phi) < 0$ for all $\phi$ in $V\setminus\{0\}$. The index can be computed analytically in terms of the spectrum of a second order differential operator with Robin boundary conditions. More precisely, integration by parts gives
	\begin{equation*}
	Q^{M}(\phi,\phi) = - \int_{M} \phi\mathcal{L}_{M}(\phi) d\mu + \int_{\partial M} \phi\left(\frac{\partial\phi}{\partial \nu} - II^{\partial \Omega}(N,N)\phi\right) d\sigma
	\end{equation*}
	where $\mathcal{L}_{M} = \Delta_{M} + Ric^{\Omega}(N,N) + |A|^2$ is the Jacobi operator of $M^n$. The boundary condition
	\begin{equation*}
	\frac{\partial\phi}{\partial \nu} = II^{\partial \Omega}(N,N)\phi
	\end{equation*}
	is an elliptic boundary condition for $\mathcal{L}_{M}$, therefore there exists a non-decreasing and diverging sequence $\lambda_{1}\leq \lambda_{2}\leq\ldots \leq \lambda_{k} \nearrow \infty$ of eigenvalues  associated to a $L^{2}(M,d\mu)$-orthonormal basis $\{\phi_k\}_{k=1}^{\infty}$ of solutions to the eigenvalue problem
	\begin{equation*}
 	\begin{cases}
		\mathcal{L}_{M}(\phi) + \lambda \phi  = 0 \quad & \text{on} \quad M^n, \\
		\frac{\partial\phi}{\partial \nu} - II^{\partial \Omega}(N,N)\phi = 0\quad & \text{on} \quad \partial M.
	\end{cases} 
	\ \ \ (*)
	\end{equation*}
	\indent The index of the free boundary minimal hypersurface $M^n$ is then equal to the number of negative eigenvalues of the system $(*)$ above (see more details in \cite{Sch}, \cite{CheFraPan} and \cite{MaxNunSmi}, Section 2). \\
	\indent The solutions of $(*)$ have a standard variational characterization: If $V_{k}$ denotes the subspace spanned by the first $k$ eigenfunctions for the above problem, then the value of the next eigenvalue $\lambda_{k+1}(\mathcal{L}_{M})$ equals the minimum of $Q^{M}$ on the $L^{2}(M,d\mu)$ orthogonal complement of $V_k$. i.e.,
	\begin{equation*}
	\lambda_{k+1}(\mathcal{L}_M) = \min_{\phi \in V_{k}^{\perp}\setminus \{0\}} \frac{Q^{M}(\phi,\phi)}{\int_{M}\phi^{2}d\mu}.
	\end{equation*}
	\indent The minimum is attained precisely by eigenfunctions of $\mathcal{L}_{M}$ associated to $\lambda_{k+1}$ and satisfying the boundary conditions in $(*)$.
	
	\section{Hodge Theory and Bochner formula for manifolds with boundary}\label{sec:hodge}
	
	\indent Let $(M^n,g)$ be a compact orientable manifold with non-empty boundary. The Hodge Theorem asserting the existence of a unique harmonic representative in every de Rham cohomology class can be extended to this setting when one assumes the appropriate boundary condition for the harmonic forms. A detailed account on this generalization of Hodge's Theorem, including an overview of its historical developments, can be found in \cite{Schw}. \\
	\indent Let $d$ denote the exterior differential on $M^n$ and let $d^{\ast}: \Omega^{p}(M)\to\Omega^{p-1}(M)$ denote the codifferential defined in terms of the Hodge star operator on $(M^n,g)$ (so that, as a result, $d^{\ast}= (-1)^{n(p+1)+1}*d* $). We define the sets
	\begin{equation*}
	\mathcal{H}^{p}_{N}(M,g) = \{\omega\in \Omega^{p}(M);\, d\omega =0, \, d^{\ast}\omega =0 \,\, \text{on} \,\, M^n \,\, \text{and} \,\, i_{\nu}\omega = 0 \,\, \text{on} \,\, \partial M\}
	\end{equation*}
	and
	\begin{equation*}
	\mathcal{H}^{p}_{T}(M,g) = \{\omega\in \Omega^{p}(M);\, d\omega =0, \, d^{\ast}\omega =0 \,\, \text{on} \,\, M^n \,\, \text{and} \,\, \nu\wedge\omega = 0 \,\, \text{on} \,\, \partial M\}.
	\end{equation*}
	\indent In other words, $\mathcal{H}^{p}_{N}(M,g)$ is the set of harmonic $p$-forms that are tangential at $\partial M$ and $\mathcal{H}^{p}_{T}(M,g)$ is the set of harmonic $p$-forms that are normal at $\partial M$.
	\begin{rmk}
	On manifolds with boundary, it is no longer true that a solution to the equation $\Delta_{p}\omega = (dd^{\ast} +d^{\ast} d)\omega=0$ is also a solution to both equations $d\omega =0$ and $d^{\ast}\omega = 0$. We use the expression ``harmonic form'' to call any differential form that is simultaneously closed and co-closed.
	\end{rmk}
	\indent The following lemma will be needed in the proof of our main theorems.
	\begin{lemm}[Cf. \cite{Schw}, Theorem 3.4.4]\label{lem:maxprinc} Let $(M^n,g)$ be a complete, connected, orientable Riemannian manifold with non-empty boundary $\partial M$. If a harmonic $p$-form vanishes identically on $U\cap\partial M^n\neq\emptyset$ for some open subset $U\subset M$, then it vanishes identically on $M^n$.
	\end{lemm}
	\indent Using the above terminology, the Hodge-de Rham theorem can be stated as follows.
	\begin{thm}\label{thm:hdr}
	Let $(M^n,g)$ be a compact orientable manifold with non-empty boundary. For every $p=0,\ldots,n$, the set of harmonic $p$-forms on $M^n$ that are tangential at $\partial M$ is isomorphic to the $p$-th cohomology group of $M^n$ with real coefficients, i.e.,
	\begin{equation*}
		\mathcal{H}^{p}_{N}(M,g) \simeq H^{p}(M;\mathbb{R}).
	\end{equation*}		
	\end{thm}
	\indent A proof of Theorem \ref{thm:hdr} can be found, for instance, in \cite{Tay} (see Chapter 5, Section 9). An elementary and elegant proof that $\mathcal{H}^{1}_{N}(M,g)$ is isomorphic to $H^{1}(M;\mathbb{R})$ is also given in \cite{Ros2}, Lemma 1.  \\
	\indent Observe that the Hodge star operator of $(M^n,g)$ gives an isomorphism between $\mathcal{H}^{p}_{N}(M,g)$ and $\mathcal{H}^{n-p}_{T}(M,g)$. Hence, we have the isomorphisms
	\begin{equation*}
		\mathcal{H}^{p}_{T}(M,g) \simeq \mathcal{H}^{n-p}_{N}(M,g) \simeq H^{n-p}(M;\mathbb{R}) \simeq H_{p}(M,\partial M;\mathbb{R}).
	\end{equation*}
	the last following by Poincar\'e-Lefschetz duality (see, for example, \cite{Hat}, Theorem 3.43). \\
	\indent Once we know that $\mathcal{H}^{1}_{T}(M,g)$ is isomorphic to $H_{1}(M,\partial M;\mathbb{R})$, it is interesting to compute the dimension of this relative homology group in terms of homology groups of $M$ and $\partial M$.
	\begin{lemm}\label{lem:relhom}
	Let $M^n$ be a compact, orientable (connected) $n$-dimensional manifold with non-empty boundary $\partial M$, $n\geq 2$. If $\partial M$ has $r\geq 1$ boundary components, then
	\begin{equation*}
	 dim H_{1}(M,\partial M;\mathbb{R}) = (r-1) + (dim H_{1}(M;\mathbb{R}) - dim Im(i_{*})),
	\end{equation*}
where $i_{*} : H_{1}(\partial M;\mathbb{R}) \rightarrow H_{1}(M;\mathbb{R})$ denotes the map between first homology groups induced by the inclusion $i: \partial M \rightarrow M$.
	\end{lemm}	
	\begin{proof}
	The definition of zero-th homology groups immediately yields $dim H_{0}(M;\mathbb{R})=1$, $dim H_{0}(\partial M;\mathbb{R})=r$ and $dim H_{0}(M,\partial M;\mathbb{R})= 0$. At that stage, a direct computation involving the last part of the long exact sequence for the pair $(M,\partial M)$,
	\begin{equation*}
	   H_{1}(\partial M;\mathbb{R}) \overset{i_{\ast}}{\rightarrow} H_{1}(M;\mathbb{R})\rightarrow H_{1}(M,\partial M;\mathbb{R})\rightarrow H_{0}(\partial M;\mathbb{R}) \rightarrow H_{0}(M;\mathbb{R}) \rightarrow H_{0}(M,\partial M;\mathbb{R})
	\end{equation*}
	\noindent proves the result.
	\end{proof}
	
	The following lemma can be proven directly, in a standard fashion, by using the representation of a compact surface with boundary as a polygon with identified edges and small open balls removed.
	
	\begin{lemm}\label{lem:twodimhom}
	Let $M^2$ be a compact, orientable surface with non-empty boundary $\partial M$. If $M$ has genus $g$ and $r\geq 1$ boundary components, then
	\begin{equation*}
	dim H_{1}(M,\partial M) = 2g + r - 1.
	\end{equation*}
	\end{lemm}
	\indent We finish this section computing the boundary term that appears when performing integration by parts of the Bochner formula for one-forms in compact manifolds with boundary. 
	 
	\begin{lemm}\label{lem:bochner}
	Let $(M^n,g)$ be a compact, orientable Riemannian manifold with boundary. 
	\begin{enumerate}
	\item{Let $\omega\in \mathcal{H}^{1}_{T}(M,g)$ be a harmonic one-form on $M^n$ that is normal at the boundary. Then
	\begin{equation*}
	\int_{M}|\nabla^{M}\omega|^2 + Ric^M(\omega,\omega)d\mu = -\int_{\partial M}H^{\partial M}|\omega|^2d\sigma.
	\end{equation*}}
   \item{Let $\omega\in \mathcal{H}^{1}_{N}(M,g)$ be a harmonic one-form on $M^n$ that is tangential at the boundary. Then
   	\begin{equation*}
   	\int_{M}|\nabla^{M}\omega|^2 + Ric^M(\omega,\omega)d\mu = -\int_{\partial M}A^{\partial M}(\omega^\sharp, \omega^\sharp)d\sigma.
   	\end{equation*}
   }
   \end{enumerate}
	\end{lemm}
	\begin{proof} 
	Since $\omega$ is harmonic, the Bochner formula for one-forms gives
	\begin{equation*}
	0 = \langle (dd^{\ast}+d^{\ast} d)\omega,\omega\rangle = -\langle\Delta_{M}\omega,\omega\rangle + Ric^M(\omega,\omega).
	\end{equation*}
	\indent Integrating by parts, we have
	\begin{equation*}
	\int_{M}|\nabla^{M}\omega|^2 + Ric^M(\omega,\omega)d\mu = \int_{\partial M}g(\nabla^{M}_\nu\omega,\omega)d\sigma
	\end{equation*}
	where $\nu$ is the outward pointing unit conormal of $\partial M$. \\
	\indent In order to prove our first assertion (part (1)), let $\{T_{1},\ldots,T_{n-1}\}$ be a local orthonormal frame on $\partial M$, so that $\{T_{1},\ldots,T_{n-1},\nu\}$ is an orthonormal basis of the tangent space of $M^n$ at the points $p$ in $\partial M$ where the frame is defined. Since by assumption
	\begin{equation*}
	0 = d^{\ast}\omega = - div_{M}\omega = - \sum_{i=1}^{n-1}g(\nabla^{M}_{T_i}\omega,T_{i}) - g(\nabla^{M}_{\nu}\omega,\nu),
	\end{equation*}
	and $\omega = \lambda\nu$ on $\partial M$ for some smooth function $\lambda$ on $\partial M$, we have
	\begin{equation*}
	g(\nabla^{M}_{\nu}\omega,\omega) = \lambda g(\nabla^{M}_{\nu}\omega,\nu) = - \lambda\sum_{i=1}^{n-1}g(\nabla^{M}_{T_i}\omega,T_i) = - \lambda^2\sum_{i=1}^{n-1}g(\nabla^{M}_{T_i}\nu,T_i) = - H^{\partial M}|\omega|^2.
	\end{equation*}
	\indent Concerning part (2), we have
	\begin{equation*}
	g(\nabla^{M}_{\nu}\omega,\omega) = g(\nabla^M_{\nu} \omega^\sharp, \omega^\sharp) =(\nabla^{M}_{\nu}\omega)(\omega^{\sharp})= (\nabla^{M}_{\omega^{\sharp}}\omega)(\nu)=g(\nabla^{M}_{\omega^{\sharp}}\omega^{\sharp},\nu)
	\end{equation*}
	where the third equality relies on the fact that $d\omega=0$, hence assuming $i_{\nu}\omega=0$ on $\partial M$ we get
	\begin{equation*}
	g(\nabla^{M}_{\omega^{\sharp}}\omega^{\sharp},\nu)= -g(\omega^{\sharp},\nabla^{M}_{\omega^{\sharp}}\nu)=-A^{\partial M}(\omega^\sharp, \omega^\sharp)
	\end{equation*}
	so that in the end $g(\nabla^{M}_{\nu}\omega,\omega)=  -A^{\partial M}(\omega^\sharp, \omega^\sharp)$
	and the conclusion follows.
	\end{proof}

	\section{The test functions obtained from harmonic one-forms}\label{sec:testf}
	
	\indent Let $\Omega$ be a domain in $\mathbb{R}^{n+1}$ with smooth boundary $\partial \Omega$. Following the notations of \cite{AmbCarSha}, let us denote by $\{\theta_1,\ldots,\theta_{n+1}\}$ a fixed orthonormal basis of $\mathbb{R}^{n+1}$. Given a compact, free boundary minimal hypersurface $M^{n}$ in $\Omega$, we want to compute the index form on the functions $\langle N\wedge \omega^{\sharp},\theta_{i}\wedge\theta_{j}\rangle$ for $1\leq i < j\leq n+1$, where $N$ is the unit normal vector field along $M^n$ and $\omega\in \mathcal{H}^{1}_{N}(M,g)$ is a harmonic one-form on $M^n$ that is normal at $\partial M$.
	\begin{prop} \label{prop:omegaN}
		Let $\Omega$ be a domain in $\mathbb{R}^{n+1}$ whose boundary has mean curvature $H^{\partial \Omega}$ and second fundamental form  $II^{\partial \Omega}$ with respect to the outward normal $\nu$. Let $M^n$ be a compact, orientable free boundary minimal hypersurface in $\Omega$. 
		\begin{enumerate}
		\item{Given a harmonic one-form $\omega$ on $M^n$ that is normal at the boundary $\partial M$, let 
		\begin{equation*}
		u_{ij} = \langle N\wedge\omega^{\sharp},\theta_{i}\wedge\theta_{j}\rangle, \quad i,j = 1,\ldots, n+1, \, i < j, 
		\end{equation*}
		\noindent denote the coordinates of $N\wedge\omega^{\sharp}$ with respect to some orthonormal basis $\{\theta_{ij}\}_{i<j}$ of $\Lambda^2\mathbb{R}^{d}$. Then
		\begin{align} \label{eqomegaN}
		\nonumber \sum_{1\leq i<j \leq n+1}Q(u_{ij},u_{ij}) = - \int_{\partial M} H^{\partial \Omega}|\omega|^2d\sigma.
		\end{align}}
	\item{Given a harmonic one-form $\omega$ on $M^n$ that is tangential at the boundary $\partial M$ and using the same notations, then
		\begin{align} 
		\nonumber \sum_{1\leq i<j \leq n+1}Q(u_{ij},u_{ij}) = - \int_{\partial M} (II^{\partial \Omega} (N,N)|\omega|^2 + II^{\partial \Omega}(\omega^\sharp, \omega^\sharp) )d\sigma.
		\end{align}}
	\end{enumerate}
	\end{prop}	
	\begin{proof}
	\indent Since $\Omega$ is flat, the index form of $M^n$ is given by
	\begin{equation*}
	Q(\phi,\phi) = \int_{M} |\nabla^{M}\phi|^2 - |A|^2\phi^2d\mu - \int_{\partial M} II^{\partial \Omega}(N,N)\phi^{2}d\sigma.
	\end{equation*}	
	\indent Following the computations in \cite{AmbCarSha}, section 3, we have 
	\begin{equation*}
	\sum_{1\leq i < j \leq n+1}Q(u_{ij},u_{ij}) = \int_{M} |D(N\wedge\omega^{\sharp})|^2 - |A|^2|N\wedge\omega^{\sharp}|^2d\mu - \int_{\partial M} II^{\partial \Omega}(N,N)|N\wedge\omega^{\sharp}|^2d\sigma.
	\end{equation*}
	\indent Clearly, $|N\wedge\omega^{\sharp}|=|\omega|$. Moreover, since the ambient curvature is identically zero,
	\begin{equation*}
	|D(N\wedge\omega^{\sharp})|^2 = |\nabla^{M}\omega|^2 - |A(\omega^{\sharp},\cdot)|^2 + |A|^2|\omega|^2 
	\end{equation*}
	and the Gauss equation for the minimal hypersurface $M^n$ reads $Ric^{M}(\omega,\omega) = - |A(\omega^{\sharp},\cdot)|^2$. The resulting formula is
	\begin{equation}
	\label{eq:ident}
	\sum_{1\leq i < j \leq n+1}Q(u_{ij},u_{ij}) = \int_{M} |\nabla^{M}\omega|^2 + Ric^{M}(\omega,\omega)d\mu - \int_{\partial M} II^{\partial \Omega}(N,N)|\omega|^2d\sigma.
	\end{equation}
	\indent By part (1) of Lemma \ref{lem:bochner},
	\begin{equation*}
	\sum_{1\leq i < j \geq n+1}Q(u_{ij},u_{ij}) = - \int_{\partial M} \left(H^{\partial M} + II^{\partial \Omega}(N,N)\right)|\omega|^2 d\sigma.
	\end{equation*}
	\indent The free boundary assumption implies that, at each point $p$ in $\partial M$, if we denote by $\{T_{1},\ldots T_{n-1}\}$ an orthonormal basis of $T_p\partial M$ so that $\{T_1,\ldots,T_{n-1},N\}$ is an orthonormal basis of $T_p\partial \Omega$, then
	\begin{align*}
	H^{\partial M} + II^{\partial \Omega}(N,N) & = \sum_{i=1}^{n-1}\langle\nabla^{M}_{T_i}\nu,T_i\rangle+ II^{\partial \Omega}(N,N) = \sum_{i=1}^{n-1}\langle D_{T_i}\nu,T_i\rangle + II^{\partial \Omega}(N,N)\\ & = \sum_{i=1}^{n-1} II^{\partial \Omega}(T_i,T_i) + II^{\partial \Omega}(N,N)  = H^{\partial \Omega}.
	\end{align*}
	The first conclusion follows. Similarly, the proof of part (2) can be completed by combining equation \eqref{eq:ident} with Lemma \ref{lem:bochner}, part (2), and noting that $A^{\partial M}(\omega^\sharp, \omega^\sharp)=II^{\partial \Omega} (\omega^\sharp, \omega^\sharp)$ by virtue of the free boundary property of $M^n$.
	\end{proof}

	\section{Proof of the main results} \label{sec:mainproof}
	
	We are now ready to present the proof of our main results, starting with Theorem A.

	\begin{proof}
	Fix $\{\theta_1,\ldots,\theta_{n+1}\}$ an orthonormal basis of $\mathbb{R}^{n+1}$. Let us assume that $M^n$ has index $k$, and denote by $\{\phi_q\}_{q=1}^{\infty}$ an $L^2(M,d\mu)$ orthonormal basis of eigenfunctions of the Jacobi operator of $M^n$ satisfying the Robin boundary conditions $(*)$. Let $\Phi$ denote the linear map defined by
	\begin{equation*}
	\begin{matrix}
    \Phi : & \mathcal{H}^{1}_{T}(M,g) & \rightarrow  & \mathbb{R}^{n(n+1)k/2} \\
           &  \omega & \mapsto & \left[\int_{M} \langle N\wedge \omega^{\sharp}, \theta_{i}\wedge\theta_{j} \rangle\phi_{q} d\mu\right],
    \end{matrix}
    \end{equation*}
    \noindent where,  $1\leq i < j \leq n+1$ and $q$ varies from $1$ to $k$. Clearly,
    \begin{equation*}
    dim \mathcal{H}^{1}_{T}(M,g) \leq dim Ker(\Phi) + \frac{n(n+1)}{2}k
    \end{equation*}
    \indent Since $\mathcal{H}^{1}_{T}(M,g)\simeq H_{1}(M,\partial M;\mathbb{R})$ (as a consequence of Theorem \ref{thm:hdr}), the result will follow once we analyse the dimension of the kernel of the map $\Phi$.\\
    \indent Let $\omega$ be an element of the kernel of the map $\Phi$. This means that all functions $u_{ij}=\langle N\wedge \omega^{\sharp}, \theta_{i}\wedge\theta_{j} \rangle$ are orthogonal to the first $k$ eigenfunctions, namely $\phi_1,\ldots, \phi_k$. Since $index(M)=k$, we must have 
    \begin{equation*}
    Q(u_{ij},u_{ij})\geq \lambda_{k+1}\int_{M}u_{ij}^2d\mu \geq 0 \quad \text{for all} \quad 1\leq i<j \leq n+1,
    \end{equation*}
    by the variational characterization of the eigenvalues for problem $(*)$. In particular, by Proposition \ref{prop:omegaN}, we have
    \begin{equation*}
    0 \leq \sum_{1\leq i< j \leq n+1}Q(u_{ij},u_{ij}) = - \int_{\partial M}H^{\partial \Omega}|\omega|^2 d\sigma.
    \end{equation*}
    \indent Since $H^{\partial \Omega} > 0$, the above inequality happens only if $|\omega|$ vanishes identically on $\partial M$. But then $\omega = 0$ on $M$, by Lemma \ref{lem:maxprinc}. Hence, if the domain is strictly mean convex, $\Phi$ has trivial kernel, and the conclusion follows. \\
    \end{proof}

	We shall now comment on the proof of Theorem F for two-convex domains in $\mathbb{R}^{n+1}$. Under this slightly stronger condition we work with test functions derived from forms $\omega \in \mathcal{H}^1_{N}(M,g)$. The process of obtaining an index estimate on $M^n$, given part (2) of Proposition \ref{prop:omegaN}, is almost identical to the proof of Theorem A so we leave the details to the reader. The direct result is the following assertion:

		\begin{thm}	\label{thm:2conv}
			Let $\Omega^{n+1}$ be a strictly two-convex domain of the $(n+1)$-dimensional Euclidean space, $n\geq 2$. Let $M^n$ be a compact, orientable properly embedded free boundary minimal hypersurface of $\Omega$. Then
			\begin{equation*}
			index(M) \geq \frac{2}{n(n+1)} dim H_{n-1}(M,\partial M;\R).	\end{equation*}
		\end{thm}
		
		Clearly, since any strictly two-convex domain is also strictly mean convex, we can take linear combinations of our estimates in Theorem A and Theorem \ref{thm:2conv} to obtain Theorem F.
		
		\
		
		The last part of this section is devoted to the proof of Theorem C. Let us remark that while for \textsl{strictly} mean convex domains the conclusion would follow at once by combining Theorem A with Lemma \ref{lem:twodimhom}, the study of the borderline case when $H^{\partial\Omega}\geq 0$ requires a more delicate, specific analysis.

	\begin{proof}
		Let us assume, by contradiction, that there exists a compact free boundary minimal surface $M^2$ in a mean convex domain $\Omega$ of $\mathbb{R}^3$ such that the opposite inequality holds:
		\begin{equation} \label{eq:ineqcontrad}
		index(M) < \frac{1}{3} (2g+r-1).
		\end{equation}
		\indent Adopting the same notation as in the proof of Theorem A, we consider the balancing maps $\Phi_N: \mathcal{H}^{1}_{N}(M,g) \rightarrow \mathbb{R}^{3k} $ and $\Phi_T: \mathcal{H}^{1}_{T}(M,g) \rightarrow \mathbb{R}^{3k}$ given by $\omega \mapsto \left[\int_{M} \langle N\wedge \omega^{\sharp}, \theta_{i}\wedge\theta_{j}\rangle\phi_{q} d\mu\right]$, where $k$ denotes the index of $M^2$. \\
		\indent Clearly,
		\begin{equation*}
		dim \mathcal{H}^{1}_{N}(M,g) \leq dim Ker(\Phi_N) + 3k \,\, \text{and} \,\, dim \mathcal{H}^{1}_{T}(M,g) \leq dim Ker(\Phi_T) + 3k. 
		\end{equation*}
		\indent Since $dim\mathcal{H}^{1}_{N}(M,g)=dim\mathcal{H}^{1}_{T}(M,g)=2g+r-1$ (see Lemma \ref{lem:twodimhom}) and we are assuming by contradiction that inequality \eqref{eq:ineqcontrad} holds, $Ker(\Phi_N)$ and $Ker(\Phi_T)$ must be both non-trivial. Let $\omega_1\neq 0$ and $\omega_2\neq 0$ be non-zero elements in $Ker(\Phi_N)$ and $Ker(\Phi_T)$, respectively. \\
		\indent As in the proof of Theorem A (invoking part (1) of Proposition \ref{prop:omegaN}), we have 
		\begin{equation*}
		0 \leq \lambda_{k+1} \int_{M}|\omega_1|^2d\mu \leq \sum_{1\leq i< j \leq 3}Q(u_{ij},u_{ij}) = - \int_{\partial M}H^{\partial M}|\omega_1|^2 d\sigma.
		\end{equation*}
		\indent Since $H^{\partial \Omega} \geq 0$ and $\omega_1\neq 0$ vanishes on a subset of $\partial M$ with no interior points (by virtue of Lemma \ref{lem:maxprinc}), the above inequalities can only happen when $H^{\partial \Omega}$ vanishes along $\partial M$, $\lambda_{k+1}=0$, and each function $u_{ij}=\langle N\wedge\omega^{\sharp},\theta_{i}\wedge\theta_{j}\rangle$ is an eigenfunction of the Jacobi operator associated to that eigenvalue, i.e., for all $i<j=1,\ldots,3$,
		\begin{align}
		\Delta_{M}\langle N\wedge\omega_1^{\sharp},\theta_{i}\wedge\theta_{j}\rangle + |A|^2\langle N\wedge\omega_1^{\sharp},\theta_{i}\wedge\theta_{j}\rangle = & 0 \quad \text{on} \quad M \label{eq:interior} \\
		\frac{\partial}{\partial \nu}\langle N\wedge\omega_1^{\sharp},\theta_{i}\wedge\theta_{j}\rangle - II^{\partial \Omega}(N,N)\langle N\wedge\omega_1^{\sharp},\theta_{i}\wedge\theta_{j}\rangle = & 0 \quad \text{on} \quad \partial M. \label{eq:boundary}
		\end{align} 
		\indent At this stage, a similar analysis can be carried out for $\omega_2\neq 0$ in $Ker(\Phi_T)$ and leads to conclude that any such one-form will necessarily satisfy equations \eqref{eq:interior} and \eqref{eq:boundary} above. \\
		
		\noindent \textit{Claim 1:} $M^2$ is totally geodesic. \\
		
		\indent Indeed, equation \eqref{eq:boundary} for $\omega_1$ and $\omega_2$ is equivalent to
		\begin{align*}
		0 & =  D_\nu N \wedge \omega_i^{\sharp} + N\wedge D_\nu \omega_i^{\sharp} - II^{\partial \Omega}(N,N)N\wedge \omega_i^{\sharp}  \\
		& =  D_\nu N \wedge \omega_i^{\sharp} + N\wedge (\nabla_{\nu} \omega_i^{\sharp} - II^{\partial \Omega}(N,N)\omega_i^{\sharp}). 
		\end{align*}
		\indent The first term is (dual to) a vector orthogonal to $M^2$, whereas the second is (dual to) a vector tangent to $M^2$. By linear independence, we conclude that, for $i=1,2$, 
		\begin{equation*}
		N\wedge(\nabla_{\nu} \omega_i^{\sharp} - II^{\partial \Omega}(N,N)\omega_i^{\sharp}) = 0 \quad \text{and} \quad D_\nu N \wedge \omega_i^{\sharp} = 0.
		\end{equation*}
		\indent The second equation means that, at boundary points where $\omega_1\neq 0$, $D_\nu N$ is a vector parallel to $\omega_1^{\sharp}$, or equivalently, tangent to $\partial M$ (because $\omega_1$ belongs to $\mathcal{H}^1_{N}(M,g)$). Similarly, at points where $\omega_2\neq 0$, $D_\nu N$ is a vector parallel to $\omega_2^{\sharp}$, or equivalently, orthogonal to $\partial M$ (because $\omega_2$ belongs to $\mathcal{H}^1_{T}(M,g)$). As a consequence, $D_\nu N = 0$ at points of $\partial M$ where both $\omega_1$ and $\omega_2$ do not vanish. Exploiting again Lemma \ref{lem:maxprinc}, we observe that the subset of $\partial M$ where either $\omega_1$ or $\omega_2$ vanishes contains no open set, because none of these one-forms vanishes identically. Therefore the complement of that set in $\partial M$ is dense and, by continuity, we conclude that $D_\nu N = 0$ everywhere on $\partial M$. \\
		\indent Thus, $\nu$ must be a principal direction at every point of $\partial M$ and the principal curvature associated to that direction is zero. Since $M^2$ is minimal, both principal curvatures actually vanish along $\partial M$. The claim now follows because any two-dimensional minimal surface in $\mathbb{R}^3$ is either totally geodesic or its second fundamental form vanishes on a discrete set of points. \\
		
		\noindent \textit{Claim 2:} $\omega_1^{\sharp}, \ \omega_{2}^{\sharp}$ are constant, non-zero vectors in $\mathbb{R}^3$. \\
		
		\indent This is a consequence of the previous claim and the equations \eqref{eq:interior}, which are satisfied by $\omega_1$ and $\omega_2$ as argued before. In fact, in view of Lemma 1 in \cite{Ros1}, $\omega_1$ and $\omega_2$ belong to the vector space
		\begin{equation*}
		\{\omega \in \Omega^{1}(M);\, \text{there exists a constant vector $T\in\mathbb{R}^3$ such that } \omega=\langle T,-\rangle\}.
		\end{equation*}
		Since $M^2$, being totally geodesic, is contained in a plane (necessarily orthogonal to the constant vector $N$), this space has dimension two and is in fact equal to the set of one-forms $\omega \in \Omega^{1}(M)$ that can be written as $\omega = \langle T,-\rangle$ for some $ T\in \mathbb{R}^3$ orthogonal to $N$. The claim follows. \\
		
		\indent Using the above claims, we conclude in particular that $\partial M$ is a collection of closed planar curves, each one tangent everywhere to the constant non-zero vector $\omega_{1}^{\sharp}$ in $\mathbb{R}^3$, which is a contradiction. Theorem C follows.
		
	\end{proof}

				\section{General ambient manifolds}\label{app:general}	
	
	In this section we state the most general result one can obtain by using the approach developed in \cite{AmbCarSha} combined with the computations presented in this paper (in particular: Lemma \ref{lem:bochner} and Proposition \ref{prop:omegaN}). Given $(\Omega^{n+1},g)$ a smooth, orientable Riemannian manifold with boundary and an isometric embedding thereof into some Euclidean space $\mathbb{R}^d$ (of possibly large dimension), and given a compact, properly embedded two-sided free boundary minimal hypersurface $M^n$ in $(\Omega^{n+1},g)$, we use as test functions for the index form the coordinates of $N\wedge \omega^{\sharp}$, $\omega$ in $\mathcal{H}^{1}_{T}(M,g)$, with respect to a fixed orthonormal basis of $\Lambda^{2}\mathbb{R}^d$. The index estimate will follow as soon as a pinching condition involving the intrinsic and extrinsic geometry of the manifold $(\Omega^{n+1},g)$ can be verified.
	\begin{thm}\label{thm:genamb}
		Let $(\Omega^{n+1},g)$ be a Riemannian manifold with boundary that is isometrically embedded in some Euclidean space $\mathbb{R}^d$. Let $M^{n}$ be a compact, orientable properly embedded free boundary minimal hypersurface of $(\Omega^{n+1},g)$. \\
		\indent Assume that for every non-zero vector field $X$ on $M^n$,
		\begin{multline*}
		\int_{M} \left[tr_M (Rm^{\Omega}(\cdot, X,\cdot,X)) + Ric^{\Omega}(N,N)|X|^2\right] d\mu 	+ \int_{\partial M} H^{\partial \Omega}|X|^2d\sigma \\
		> \int_{M} \left[(|II^{\Omega}(\cdot,X)|^2 -|II^{\Omega}(X,N)|^2) + (|II^{\Omega}(\cdot,N)|^2-|II^{\Omega}(N,N)^2|)|X|^2\right]d\mu 
		 \end{multline*}
		\noindent where $Rm^{\Omega}$ denotes the Riemann curvature tensor of $\Omega^{n+1}$, $II^{\Omega}$ denotes the second fundamental form of $\Omega^{n+1}$ in $\mathbb{R}^d$ and $N$ is a local unit normal vector field on $M^n$. \\
		\indent Then
		\begin{equation}
		index(M) \geq \frac{2}{d(d-1)}dim H_{1}(M,\partial M;\mathbb{R}).
		\end{equation} 
	\end{thm}
	\indent The computations presented in \cite{AmbCarSha} show that the above pinching condition is verified in any strictly mean convex domain inside the ambient manifolds considered there. In particular, the index estimate above holds true for all compact properly embedded free boundary minimal hypersurfaces in strictly mean convex domains of the compact rank one symmetric spaces (round spheres and projective spaces over $\mathbb{R},\mathbb{C},\mathbb{H}$ and the Cayley plane endowed with their canonical metrics). Also, let us explicitly remark that our method allows to obtain an effective index estimate for free boundary minimal hypersurfaces in strictly mean convex domains of flat tori. \\
	
	\indent We can of course also consider the generalization of Theorem \ref{thm:2conv} (hence of Theorem F) to two-convex domains in Riemannian manifolds. The corresponding statement is as follows:
	
	\begin{thm}\label{thm:combgen}
		Let $(\Omega^{n+1},g)$ be a Riemannian manifold with boundary that is isometrically embedded in some Euclidean space $\mathbb{R}^d$. Let $M^{n}$ be a compact, orientable properly embedded free boundary minimal hypersurface of $(\Omega^{n+1},g)$. \\
		\indent Assume that for every non-zero vector field $X$ on $M^n$,
		\begin{multline*}
		\int_{M} \left[tr_M (Rm^{\Omega}(\cdot, X,\cdot,X)) + Ric^{\Omega}(N,N)|X|^2\right] d\mu + \int_{\partial M} II^{\partial \Omega}(N,N)|X|^2 + II^{\partial \Omega} (X,X)\, d\sigma \\
		> \int_{M} \left[(|II^{\Omega}(\cdot,X)|^2 -|II^{\Omega}(X,N)|^2) + (|II^{\Omega}(\cdot,N)|^2-|II^{\Omega}(N,N)^2|)|X|^2\right]d\mu 
		\end{multline*}
		\noindent where $Rm^{\Omega}$ denotes the Riemann curvature tensor of $\Omega^{n+1}$, $II^\Omega$ denotes the second fundamental form of $\Omega^{n+1}$ in $\mathbb{R}^d$, $II^{\partial \Omega}$ denotes the second fundamental form of $\partial \Omega$ in $\Omega$, and $N$ is a local unit normal vector field on $M^n$. \\
		\indent Then for any $\alpha \in [0,1]$
		\begin{equation}
		index(M) \geq \frac{2}{d(d-1)}\left(\alpha dim H_{1}(M,\partial M;\mathbb{R}) + (1-\alpha) dim H_{n-1} (M,\partial M ; \R)\right). 
		\end{equation} 
	\end{thm}
	Once again, the above theorem applies, as a special case, if $\Omega^{n+1}$ is a domain with strictly two-convex boundary inside any of the Riemannian manifolds considered in \cite{AmbCarSha}.   
	
	\

\bibliographystyle{amsbook}

\end{document}